\documentclass [10pt]{article}
\usepackage{amsfonts,tikz}
\usepackage{amsmath, amsthm, amssymb}
\usepackage{graphicx}
\usepackage{indentfirst}
\usepackage[T1]{fontenc}
\usepackage[margin=1in]{geometry}
\usepackage{color}
\usepackage{slashbox}

\numberwithin{equation}{section}

\newcommand{\Z}{\mathbb{Z}}
\newcommand{\C}{\mathbb{C}}

\newcommand{\N}{\mathbb{N}}
\newcommand{\R}{\mathbb{R}}

\newcommand{\Hc}{\mathcal{H}}

\newcommand{\etaq}[2]{\eta^{#2}( #1 z)}
\newcommand{\floor}[1]{\left\lfloor #1 \right\rfloor}
\newcommand{\ceiling}[1]{\left\lceil #1 \right\rceil}

\DeclareMathOperator{\SL}{SL}
\DeclareMathOperator{\sgn}{sgn}

\newtheorem{theorem}{Theorem}[section]
\newtheorem{conjecture}[theorem]{Conjecture}
\newtheorem{coro}[theorem]{Corollary}
\newtheorem{defn}[theorem]{Definition}
\newtheorem{example}[theorem]{Example}

\newtheorem{prop}[theorem]{Proposition}

\title{Proceedings Paper for REU Project Involving Counting Eta-Quotients}
\author{Allison Arnold-Roksandich* and Rodney Keaton\\Advisor: Kevin James}
\date{July 2013}

\begin{document}
\maketitle
\begin{abstract}
It is known that all modular forms on $\SL_2(Z)$ can be expressed as a rational function in $\etaq{}{}$, $\etaq{2}{}$ and $\etaq{4}{}$. By using a theorem by Gordon, Hughes, and Newman, and calculating the order of vanishing, we can compute the $\eta$-quotients for a given level. Using this count, knowing how many $\eta$-quotients are linearly independent and using the dimension formula, we can figure out how the $\eta$-quotients span higher levels. In this paper, we primarily focus on the case where $N=p$ a prime, and some discussion for non-prime indicies.
\end{abstract}
\section{Introduction}
Modular forms and cusp forms encode important arithmetic information, and are therefore important to study. However, it is, in general, hard to create concrete examples of modular forms and cusp forms. $\eta$-quotients are an easy concrete way to look at modular and cusp forms. Furthermore, $\eta$-quotients are better studied.

\begin{theorem}\cite[Thm. 1.67]{Ono1}
Every modular form on $\SL_2(\Z)$ may be expressed as a rational function in $\eta(z)$, $\eta(2z)$, and $\eta(4z)$.
\end{theorem}

This theorem is the primary motivation behind this paper. The goal is to look at higher levels, and count the number of $\eta$-quotients for the given level and compare the span of these $\eta$-quotients with the modular form space. In particular, this paper focuses on $\Gamma_0(p)$ and $\Gamma_1(p)$ where $p$ is a prime. This paper also discusses the needed congruence conditions that need to be satisfied for cases $N=pq$, square-free $N$. Finally, there is two results for the space $M_k(4p,\chi )$.
\section{Modular Forms}\label{ModForm}
In this section we present some definitions and basic facts from the theory of modular forms. For further details, the interested reader is referred to \cite{Kob1}.

\begin{defn}The \emph{modular group}, denoted $\SL_2(\Z )$, is the group of all matrices of determinant 1 and having integral entries.\end{defn} 
Note, the modular group is generated by the matrices
$$S:=\begin{pmatrix}0 & -1\\1&0\end{pmatrix},T:=\begin{pmatrix}1&1\\0&1\end{pmatrix}.$$
The modular group acts on the upper half plane, $\Hc = \{ x+iy | x,y\in \R, y>0\}$, with linear fractional transformations \[\begin{pmatrix}a & b\\c&d\end{pmatrix}z = \frac{az+b}{cz+d}.\]
Furthermore, if we define $\mathcal{H}^{*}$ to be the set $\mathcal{H}\cup\mathbb{Q}\cup\{i\infty\}$, then the action of $\SL_2(\mathbb{Z})$ on $\mathcal{H}$ extends to an action on $\mathcal{H}^{*}$. This brings us to our next definition. 
\begin{defn}
Let $\Gamma\leq\SL_2(\mathbb{Z})$ and define an equivalence relation on $\mathbb{Q}\cup\{\infty\}$ by $z_1\sim z_2$ if there is a $\gamma\in\Gamma$ such that $\gamma\cdot z_1=z_2$. We call each equivalence class under this relation a \emph{cusp of} $\Gamma$.
\end{defn}
Now, for an integer $k$ and a function $f:\Hc^{*}\rightarrow\mathbb{C}$ and a $\gamma=\begin{pmatrix}a&b\\c&d\end{pmatrix}\in\SL_2(\mathbb{Z}$ we define the weight $k$ slash operator by
$$f|_k\gamma(z)=(cz+d)^{-k}f(\gamma\cdot z).$$
Note, we will often suppress the weight from the notation when it is clear from context, or irrelevant for our purposes.

We are now prepared to define the objects which will be of primary interest to us.
\begin{defn} A function $f: \Hc^{*} \rightarrow \C$ is called a \emph{weakly modular function} of weight $k$ and level $\Gamma$ if 
\begin{enumerate}
\item $f$ is holomorphic on $\Hc$,
\item $f$ is modular, i.e., for every $\gamma \in \Gamma$ and $z\in\Hc$ we have $f|\gamma(z)=f(z)$,
\item $f$ is meromorphic at each cusp of $\Gamma$.
\end{enumerate}
Furthermore, if we replace condition 3 by $f$ is holomorphic at each cusp of $\Gamma$ then we call $f$ a modular form. If we further replace condition 3 with $f$ vanishes at each cusp of $\Gamma$ then we call $f$ a cusp form.
\end{defn}
There are only certain specific subgroups of $\SL_2(\Z)$ which we will need for our purposes. They are 
\begin{align*}
\Gamma_0 (N) &= \left\lbrace \begin{pmatrix}a & b\\c&d\end{pmatrix} \in \SL_2(\Z)\left| \begin{pmatrix}a & b\\c&d\end{pmatrix} \equiv \begin{pmatrix}* & *\\0&*\end{pmatrix}\right. \pmod{N} \right\rbrace\\
\Gamma_1 (N) &= \left\lbrace \begin{pmatrix}a & b\\c&d\end{pmatrix} \in \SL_2(\Z)\left| \begin{pmatrix}a & b\\c&d\end{pmatrix} \equiv \begin{pmatrix}1 & *\\0&1\end{pmatrix}\right. \pmod{N} \right\rbrace\\
\end{align*}
We refer to each of these subgroups as a congruence subgroup of level $N$. Note, if $N=1$, then $\Gamma_0(N)=\Gamma_1(N)=\SL_2(\mathbb{Z})$.

Consider a form, $f$, of congruence level $N$. We will make precise what we mean by a function being ``holomorphic at a cusp''. First, consider the cusp $\{i\infty\}$, which we call ``the cusp at $\infty$''. Note, the matrix $T$ from above is an element of $\Gamma_1(N)$ hence $\Gamma_0(N)$ for every $N$, and as our function satisfies condition 2, we have
$f(Tz)=f(z+1)=f(z)$, i.e., our function is periodic. It is a basic fact from complex analysis that such a function has a Fourier expansion of the form 
$$f(z)=\sum_{n=-\infty}^{\infty}a_nq^n\text{, where }q:=e^{2\pi iz}.$$
Using this, we say that $f$ is meromorphic at $\{i\infty\}$ if there is some $c<0$ such that $a_n=0$ for all $n<c$. We say that $f$ holomorphic at $\{i\infty\}$ if $a_n=0$ for all $n<0$, and we say that $f$ vanishes at $\{i\infty\}$ if $a_n=0$ for all $n\leq 0$.

Note, it is not hard to see that for every rational number $r$ there is some $\gamma\in\SL_2(\mathbb{Z})$ such that $\gamma\cdot i\infty =r$, i.e., $\SL_2(\mathbb{Z})$ acts transitively on the set $\mathbb{Q}\cup\{i\infty\}$.Using such a $\gamma$, we will say that $f$ is meromorphic (holomorphic, vanishes, resp.) if $f|\gamma$ is meromorphic (holomorphic, vanishes, resp.) at $\{i \infty\}$.

Now, we set some notation which we will use throughout. For $\Gamma\leq\SL_2(\mathbb{Z})$ we denote the space of weakly modular functions (modular forms, cusp forms, resp.) of level $\Gamma$ and weight $k$ by $M^{!}_k(\Gamma)$ ($M_k(\Gamma),S_k(\Gamma)$, resp.). Note, the spaces $S_k(\Gamma)\leq M_k(\Gamma)$ are finite dimensional complex vector spaces.

Throughout, we will also need the notion of a modular form with an associated character. To this end, let $\chi$ be a Dirichlet character modulo $N$, i.e.,
$$\chi:\left(\mathbb{Z}/ N\mathbb{Z}\right)^{\mathsf{x}}\rightarrow\mathbb{C}^{\mathsf{x}}$$
is a homomorphism. Furthermore, if we let $c$ be the minimal integer such that $\chi$ factors through $\left(\mathbb{Z}/c\mathbb{Z}\right)^{\mathsf{x}}$, then we say $\chi$ has conductor $c$. 

Let $f\in M_k(\Gamma_1(N))$ and suppose further that $f$ satisfies
$$f|\gamma(z)=\chi(d)f(z)\text{, for all }\gamma=\begin{pmatrix}a&b\\c&d\end{pmatrix}\in\Gamma_0(N).$$
Then, we say that $f$ is a modular form of level $N$ and character $\chi$, and we denote the space of such functions by $M_k(N,\chi)$. Note, this is defined similarly for weakly modular functions and cusp forms.

It is well known that we have the following decomposition
$$M_k(\Gamma_1(N))=\!\!\!\!\!\bigoplus_{\chi\!\!\!\!\!\pmod{N}}\!\!\!\!\!M_k(N,\chi),$$
where our direct sum is over all Dirichlet characters modulo $N$. Decomposing further we have
$$M_k(N,\chi)=S_k(N,\chi)\oplus E_k(N,\chi)^{\perp},$$
into the space of cusp forms and its orthogonal complement, denoted $E_k(N,\chi)$, which we call the Eisenstein subspace.
\section{Dimension Formulas}
In this section we present formulas for the dimension of spaces of cusp and modular forms. After presenting the general formula we will specialize to the cases that arise in the study of $\eta$ quotients. For more details regarding dimension formulas, the interested reader is referred to \cite{Ste1}.
\subsection{The dimension formula for level 1}
In this short section we restrict to the case that $N=1$.

Define
$$\gamma_4(k)=\left\{\begin{array}{lcr}-\frac{1}{4}&\text{if} & k\equiv 2\pmod{4},\\\frac{1}{4} &\text{if}& k\equiv 0\pmod{4},\\0&\text{if}&k\equiv 1\pmod{2},\end{array}\right.$$
$$\gamma_3(k)=\left\{\begin{array}{lcr}-\frac{1}{3} &\text{if}& k\equiv 2\pmod{3},\\\frac{1}{3} &\text{if}& k\equiv 0\pmod{3},\\0&\text{if}&k\equiv 1\pmod{3}.\end{array}\right.$$
Then our dimension formula is given by
$$\dim S_k(\SL_2(\mathbb{Z}))-\dim M_{2-k}(\SL_2(\mathbb{Z}))=\frac{k-1}{12}-\frac{1}{2}+\gamma_4(k)+\gamma_3(k).$$
In the level 1 case we can write down a few well known examples.
\begin{example}
Let $k=-2$. Then, 
$$\dim S_{-2}(\SL_2(\mathbb{Z}))-\dim M_{4}(\SL_2(\mathbb{Z}))=-\dim M_{4}(\SL_2(\mathbb{Z}))=-1,$$
i.e., $\dim M_{4}(\SL_2(\mathbb{Z}))=1$. A basis element here is given by,
$$G_4(z)=\sum'_{c,d\in\mathbb{Z}}\frac{1}{(cz+d)^4},$$
where the prime on the summation indicates that we do not allow $c,d$ simultaneously zero. Note, this function is called the weight 4 Eisenstein series of level 1.
\end{example}
\begin{example}
Let $k=12$. Then,
$$\dim S_{12}(\SL_2(\mathbb{Z}))-\dim M_{-10}(\SL_2(\mathbb{Z}))=\dim S_{12}(\SL_2(\mathbb{Z}))=1.$$
In this case a basis element for $ S_{12}(\SL_2(\mathbb{Z}))$ is given by the function $\Delta(z)$ which satisfies
$$(2\pi)^{-12}\Delta(z)=\eta^24(z)=q\prod_{n=1}^{\infty}(1-q^n)^24=\sum_{n=1}^{\infty}\tau(n)q^n,$$
where $\tau(n)$ is Ramanujan's $\tau$ function. Note, we call $\Delta(z)$ the discriminant modular form.

In order to determine the dimension of the space of modular forms of weight 12, we now set $k=-10$. Then,
$$\dim S_{-10}(\SL_2(\mathbb{Z}))-\dim M_{12}(\SL_2(\mathbb{Z}))=-\dim M_{12}(\SL_2(\mathbb{Z}))=-2,$$
i.e., $\dim M_{12}(\SL_2(\mathbb{Z}))=2$. We know $\Delta(z)\in  M_{12}(\SL_2(\mathbb{Z}))$ and the second basis element is just as in the previous example, i.e., the weight 12 Eisenstein series which is given by
$$G_{12}(z):=\sum'_{c,d\in\mathbb{Z}}\frac{1}{(cz+d)^{12}}.$$
\end{example}

\subsection{The dimension formula for level $\Gamma_0(p)$}\label{Gamma0}
In this section we present a formula for the dimension of  $E_k(\Gamma_0(p))$ and $S_k(\Gamma_0(p))$ for $p\geq5$ a rational prime.

First, we set
$$\mu_{0,2}(p)=\left\{\begin{array}{lcl}0&\text{if}&p\equiv 3\pmod{4}\\2&\text{if}&p\equiv 1\pmod{4},\end{array}\right.$$
$$\mu_{0,3}(p)=\left\{\begin{array}{lcl}0&\text{if}&p\equiv 2\pmod{3}\\2&\text{if}&p\equiv 1\pmod{3}.\end{array}\right.$$
Then define
$$g_0(p)=\frac{p+1}{12}-\frac{\mu_{0,2}(p)}{4}-\frac{\mu_{0,3}(p)}{3}.$$
Using this we have $\dim S_2(\Gamma_0(p))=g_0(p)$ and $\dim E_2(\Gamma_0(p))=1$,
and for $k\geq 4$ even we have $\dim E_k(\Gamma_0(p))=2$ and
\begin{align*}
\dim S_k(\Gamma_0(p))=(k-1)(g_0(p)-1)+(k-2)+\mu_{0,2}(p)\lfloor k/4\rfloor+\mu_{0,3}(p)\lfloor k/3\rfloor.
\end{align*}
From this, we see that our formula depends on the congruence class which $k$ and $p$ lie in modulo 12, so compiling these different congruences together we have the following table. Note, we are assuming that $k>2$.
\begin{center}
\begin{tabular}{|c|*{4}{c|}}\hline
\multicolumn{5}{|c|}{$\dim S_k(\Gamma_0(p))$} \\
\hline
\backslashbox{$k(12)$}{$p(12)$}&1&5&7&11\\
\hline
0&$\frac{(p+1)(k-1)+2}{12}$&$\frac{(p+1)(k-1)-6}{12}$&$\frac{(p+1)(k-1)-4}{12}$&$\frac{(p+1)(k-1)-12}{12}$\\
\hline
1&0&0&0&0\\
\hline
2&$\frac{(p+1)(k-1)-26}{12}$&$\frac{(p+1)(k-1)-18}{12}$&$\frac{(p+1)(k-1)-20}{12}$&$\frac{(p+1)(k-1)-12}{12}$\\
\hline
3&0&0&0&0\\
\hline
4&$\frac{(p+1)(k-1)-6}{12}$&$\frac{(p+1)(k-1)-6}{12}$&$\frac{(p+1)(k-1)-12}{12}$&$\frac{(p+1)(k-1)-12}{12}$\\
\hline
5&0&0&0&0\\
\hline
6&$\frac{(p+1)(k-1)-10}{12}$&$\frac{(p+1)(k-1)-18}{12}$&$\frac{(p+1)(k-1)-4}{12}$&$\frac{(p+1)(k-1)-12}{12}$\\
\hline
7&0&0&0&0\\
\hline
8&$\frac{(p+1)(k-1)-14}{12}$&$\frac{(p+1)(k-1)-6}{12}$&$\frac{(p+1)(k-1)-20}{12}$&$\frac{(p+1)(k-1)-12}{12}$\\
\hline
9&0&0&0&0\\
\hline
10&$\frac{(p+1)(k-1)-18}{12}$&$\frac{(p+1)(k-1)-18}{12}$&$\frac{(p+1)(k-1)-12}{12}$&$\frac{(p+1)(k-1)-12}{12}$\\
\hline
11&0&0&0&0\\
\hline
\end{tabular}
\end{center}
\subsection{The dimension formula for $\Gamma_0(p)$ with quadratic character}
In this section we will consider the case that our level is $\Gamma_0(p)$ for some rational prime $p$ and that our associated character is quadratic. Note, at the end of the section we compile all of our computations together in a table for convenience.

In section \ref{Gamma0} we considered the trivial character case, so we now set $\chi(\cdot)=\left(\frac{\cdot}{p}\right)$.

We must compute the summations
$$\sum_{x\in A_4(p)}\chi(x),\sum_{x\in A_3(p)}\chi(x).$$
First, we will consider $\sum_{x\in A_4(p)}\chi(x)$. This is clearly zero if $A_4(p)$ is empty, which occurs precisely when $p\equiv 3\pmod{4}$. Also, it is immediate that our summation equals $1$ when $p=2$.  Now suppose $p\equiv 1\pmod{4}$. Then $\# A_4(p)=2$. Note, if $r\in A_4(p)$ then $-r\in A_4(p)$ and  $\chi(r)=\chi(-r)$ since $\chi(-1)=1$. Furthermore, it is not hard to see that $\chi(r)=1$ if and only if there is an element of order 8 in $\left(\mathbb{Z}/p\mathbb{Z}\right)^{\mathsf{x}}$, i.e., $p\equiv 1\pmod{8}$. Thus, we have
$$\sum_{x\in A_4(p)}\chi(x)=\left\{\begin{array}{lcl}1&\text{if}&p=2\\0&\text{if}&p\equiv 3\pmod{4}\\2&\text{if}&p\equiv 1\pmod{8}\\-2&\text{if}&p\equiv 5\pmod{8}.\end{array}\right.$$

Now we consider the summation $\sum_{x\in A_3(p)}\chi(x)$. Similar to above we have that $A_3(p)$ is empty if $p\equiv 2\pmod{3}$, in which case our summation is zero. Also, if $p=3$ then our summation is $1$. Now, suppose that $p\equiv 1\pmod{3}$. Note, it is immediate that if $r\in A_3(p)$ then so is $r^2$. Similar to the previous situation, we have that $\chi(r)=1$ if and only if there is an element of order 6 in $\left(\mathbb{Z}/p\mathbb{Z}\right)^{\mathsf{x}}$, i.e., $p\equiv 1\pmod{6}$. Note, that as $p$ is prime, it follows that $p\equiv 1\pmod{6}$ is equivalent to $p\equiv 1\pmod{3}$. Thus, we have
$$\sum_{x\in A_3(p)}\chi(x)=\left\{\begin{array}{lcl}1&\text{if}&p=3\\0&\text{if}&p\equiv 2\pmod{3}\\2&\text{if}&p\equiv 1\pmod{3}.\end{array}\right.$$

To summarize, we combine our calculations from above to obtain the following table
\begin{center}
\begin{tabular}{|c|*{4}{c|}}\hline
\multicolumn{5}{|c|}{$\dim S_k(p,\left(\frac{\cdot}{p}\right))$} \\
\hline
\backslashbox{$k(12)$}{$p(24)$}&1&5&7&11\\
\hline
0&$\frac{(k-1)(p+1)+8}{12}$&$\frac{(k-1)(p+1)-12}{12}$&0&0\\
\hline
1&0&0&$\frac{(k-1)(p+1)}{12}$&$\frac{(k-1)(p+1)-6}{12}$\\
\hline
2&$\frac{(k-1)(p+1)-20}{12}$&$\frac{(k-1)(p+1)}{12}$&0&0\\
\hline
3&0&0&$\frac{(k-1)(p+1)+2}{12}$&$\frac{(k-1)(p+1)-6}{12}$\\
\hline
4&$\frac{(k-1)(p+1)}{12}$&$\frac{(k-1)(p+1)-12}{12}$&0&0\\
\hline
5&0&0&$\frac{(k-1)(p+1)-14}{12}$&$\frac{(k-1)(p+1)-6}{12}$\\
\hline
6&$\frac{(k-1)(p+1)-4}{12}$&$\frac{(k-1)(p+1)}{12}$&0&0\\
\hline
7&0&0&$\frac{(k-1)(p+1)}{12}$&$\frac{(k-1)(p+1)-6}{12}$\\
\hline
8&$\frac{(k-1)(p+1)-4}{12}$&$\frac{(k-1)(p+1)-12}{12}$&0&0\\
\hline
9&0&0&$\frac{(k-1)(p+1)+2}{12}$&$\frac{(k-1)(p+1)-6}{12}$\\
\hline
10&$\frac{(k-1)(p+1)-12}{12}$&$\frac{(k-1)(p+1)}{12}$&0&0\\
\hline
11&0&0&$\frac{(k-1)(p+1)-14}{12}$&$\frac{(k-1)(p+1)-6}{12}$\\
\hline
\end{tabular}
\end{center}
\begin{center}
\begin{tabular}{|c|*{4}{c|}}\hline
\multicolumn{5}{|c|}{$\dim S_k(p,\left(\frac{\cdot}{p}\right))$} \\
\hline
\backslashbox{$k(12)$}{$p(24)$}&13&17&19&23\\
\hline
0&$\frac{(k-1)(p+1)-4}{12}$&$\frac{(k-1)(p+1)}{12}$&0&0\\
\hline
1&0&0&$\frac{(k-1)(p+1)}{12}$&$\frac{(k-1)(p+1)-6}{12}$\\
\hline
2&$\frac{(k-1)(p+1)-8}{12}$&$\frac{(k-1)(p+1)-12}{12}$&0&0\\
\hline
3&0&0&$\frac{(k-1)(p+1)+2}{12}$&$\frac{(k-1)(p+1)-6}{12}$\\
\hline
4&$\frac{(k-1)(p+1)-12}{12}$&$\frac{(k-1)(p+1)}{12}$&0&0\\
\hline
5&0&0&$\frac{(k-1)(p+1)-14}{12}$&$\frac{(k-1)(p+1)-6}{12}$\\
\hline
6&$\frac{(k-1)(p+1)+8}{12}$&$\frac{(k-1)(p+1)-12}{12}$&0&0\\
\hline
7&0&0&$\frac{(k-1)(p+1)}{12}$&$\frac{(k-1)(p+1)-6}{12}$\\
\hline
8&$\frac{(k-1)(p+1)-20}{12}$&$\frac{(k-1)(p+1)}{12}$&0&0\\
\hline
9&0&0&$\frac{(k-1)(p+1)+2}{12}$&$\frac{(k-1)(p+1)-6}{12}$\\
\hline
10&$\frac{(k-1)(p+1)}{12}$&$\frac{(k-1)(p+1)-12}{12}$&0&0\\
\hline
11&0&0&$\frac{(k-1)(p+1)-14}{12}$&$\frac{(k-1)(p+1)-6}{12}$\\
\hline
\end{tabular}
\end{center}
\section{$\eta$-Quotients}\label{EtaQuotient}
In this section we introduce the eta function and present some results relating these to modular forms. For further details regarding eta functions, the interested reader is referred to \cite{Kohler}

\begin{defn}\emph{Dedekind's eta-function} is defined by \[\eta (z) := q^{1/24}\prod_{n=1}^\infty (1-q^n),\qquad \text{where $q=e^{2\pi iz}$}.\]\end{defn}

The eta function satisfies the following transformation properties with respect to our matrices $S,T$ defined in Section \ref{ModForm}
$$\eta(Sz)=\eta(-z^{-1})=\sqrt{-iz}\eta(z),\text{ }\eta(Tz)=\eta(z+1)=e^{\frac{2\pi i}{24}}\eta(z).$$
Combining these we can deduce the following general transformation formula
$$\eta(\gamma z)=\epsilon(\gamma)(cz+d)^{\frac{1}{2}}\eta(z)\text{ for all } \gamma=\begin{pmatrix}a&b\\c&d\end{pmatrix}\in\SL_2(\mathbb{Z}),$$
where 
$$\epsilon(\gamma)=\left\{\begin{array}{lcl}\left(\frac{d}{|c|}\right)e^{2\pi i\left(\frac{(a+d)c-bd(c^2-1)-3c}{24}\right)}&\text{if}&c\text{ odd}\\(-1)^{\frac{1}{4}(\sgn(c)-1)(\sgn(d)-1)}\left(\frac{d}{|c|}\right)e^{2\pi i\left(\frac{(a+d)c-bd(c^2-1)+3d-3-3cd}{24}\right)}&\text{if}&c\text{ even}\end{array}\right.$$
and $\sgn(x)=\frac{x}{|x|}$. In addition to the eta function $\eta(z)$, we will also need to consider the related function $\eta(\delta z)$ for a positive integer $\delta$. If we set $f(z)=\eta(\delta z)$ then $f(z)$ satisfies
$$f(\gamma z)=\epsilon\left(\begin{pmatrix}a&\delta b\\c/\delta &d\end{pmatrix}\right)(cz+d)^{\frac{1}{2}}f(z)\text{, for all }\gamma=\begin{pmatrix}a&b\\c&d\end{pmatrix}\in\Gamma_0(\delta).$$
Finally, we will need the following transformation
$$f(Tz)=e^{\frac{2\pi i\delta}{24}}f(z).$$
Notice, that this function is ``almost'' a modular form. With this in mind, we consider certain products of these functions with the goal of eliminating the ``almost''. This brings us to the following definition.
\begin{defn}A function of the form \[f(z) = \prod_{\delta |N} \eta(\delta z)^{r_\delta}\] for $N\geq 1$ and all $r_\delta\in\Z$ is called an $\eta$-quotient.\end{defn}
We will be interested in when these $\eta$-quotients are modular forms. We have the following theorem due to Gordon, Hughes, and Newman which partially answers this question.
\begin{theorem}\label{GHNew}\cite[Thm. 1.64]{Ono1}
Define the $\eta$-quotient
$$f(z)= \prod_{\delta |N} \eta^{r_{\delta}}(\delta z),$$ 
 and set $k=\frac{1}{2}\sum_{\delta |N} r_{\delta} \in\Z$. Suppose our exponents $(r_{1},\dots,r_N)$ satisfy
\begin{align*}
\sum_{\delta |N} \delta r_{\delta} &\equiv 0\!\!\!\!\! \pmod{24},\\
\sum_{\delta |N} \frac{N}{\delta} r_{\delta} &\equiv 0\!\!\!\!\! \pmod{24}.
\end{align*}
Then
 $$
 f|_k\gamma(z) = \chi(d)f(z)
 $$
 for all $\begin{pmatrix}a& b\\c&d \end{pmatrix}\in \Gamma_0(N)$, where $\chi(n)=\left(\frac{(-1)^ks}{n}\right)$ with $s=\prod_{\delta |N}\delta^{r_{\delta}}$
\end{theorem}
This theorem provides conditions on when an $\eta$-quotient is a weakly modular functions. However, to answer the question of when an $\eta$-quotient is a modular form we need the following theorem which provides information concerning the order of vanishing at the cusps of $\Gamma_0(N)$.
\begin{theorem}\label{Cvanishing}\cite[Thm. 1.65]{Ono1}
Let $f(z)$ be an $\eta$-quotient satisfying the conditions in the previous theorem. Let $c,d,\in \N$ with $d|N$ and $(c,d)=1$. Then, the order of vanishing of $f(z)$ at the cusp $\frac{c}{d}$ is 
 $$
 \frac{N}{24}\sum_{\delta |N}\frac{(d,\delta)^2r_{\delta}}{(d,\frac{N}{d})d\delta}
 $$
\end{theorem}
We shall denote the order of vanishing for $d |N$ as $v_d$.

\section{Results}
\subsection{$N=p$}
For this section, we shall assume that the level of the modular forms being discussed is always a prime number, $p$. For this section, Theorems \ref{GHNew} and \ref{Cvanishing} imply that we need to satisfy the following: \begin{align*}
\frac{1}{2}(r_1+r_p) &=k\\
r_1 + pr_p &\equiv 0 \pmod{24}\\
pr_1 + r_p &\equiv 0 \pmod{24}\\
v_1 = \frac{1}{24}(pr_1 + r_p) &\geq 0\\
v_p = \frac{1}{24}(r_1+pr_p) &\geq 0
\end{align*}
We start the discussion for counting $\eta$-quotients at level $\Gamma_1(p)$ by looking at possible conditions on $k$. 
\begin{theorem}\label{kcond}
$\eta$-quotients which satisfy the following conditions:
\begin{enumerate}
\item $p\equiv 11,23 \pmod{24}$
\item $p\equiv 2, 17 \pmod{24}$ and $k \equiv 0 \pmod{4}$
\item $p\equiv 3,7,19 \pmod{24}$ and $k \equiv 0 \pmod{3}$
\item $p\equiv 5 \pmod{24}$ and $k \equiv 0 \pmod{2}$
\item $p\equiv 13 \pmod{24}$ and $k \equiv 0 \pmod{6}$
\item $p\equiv 1 \pmod{24}$ and $k \equiv 0 \pmod{12}$
\end{enumerate}
are weakly holomorphic modular functions of level $\Gamma_1(p)$.
\end{theorem}
\begin{proof}
($\rightarrow$) Suppose that $\etaq{}{r_1}\etaq{p}{r_p}\in M_k (\Gamma_1(p))$. Then we have \begin{align}
24v_1 &= 2k + (p-1)r_1.\label{kcondeq2}\\
24v_p &= 2kp + (1-p)r_1 \label{kcondeq1}
\end{align} We note that we can combine both equations to get \begin{equation}\label{combkcond}
24(v_1 + v_p) = 2k (p+1).
\end{equation}
We consider the 10 possible cases for primes: \begin{itemize}
\item Case 1: Suppose $p\equiv 1\pmod{24}$. Then we note that we can plug 1 in for $p$ in (\ref{combkcond}) to get that \[24 (v_1 + v_p) = 4k.\] This tells us that $k$ must be a multiple of 6. However, when we plug 1 in for $p$ in (\ref{kcondeq1}) and (\ref{kcondeq2}), we get that $24|2k$. Therefore, it must be that $k\equiv 0 \pmod{12}$.
\item Case 2: Suppose $p=2$. Then we note that we can plug 2 in for $p$ in (\ref{combkcond}) to get that \[24 (v_1 + v_p) = 6k.\] This tells us that $k$ must be a multiple of 4. When we plug 2 in for $p$ and $4n$ in for $k$ in (\ref{kcondeq1}) and (\ref{kcondeq2}), we see that there exists a sufficient $r_1$ such that $24|(16n\mp r_1)$. Therefore, it must be that $k\equiv 0 \pmod{4}$.
\item Case 3: Suppose $p=3$. Then we note that we can plug 3 in for $p$ in (\ref{combkcond}) to get that \[24 (v_1 + v_p) = 8k.\] This tells us that $k$ must be a multiple of 3. When we plug 3 in for $p$ and $3n$ in for $k$ in (\ref{kcondeq1}) and (\ref{kcondeq2}), we see that there exists a sufficient $r_1$ such that $24|(18n\mp 2r_1)$. Therefore, it must be that $k\equiv 0 \pmod{3}$.
\item Case 4: Suppose $p\equiv 5\pmod{24}$. Then we note that we can plug 5 in for $p$ in (\ref{combkcond}) to get that \[24 (v_1 + v_p) = 12k.\] This tells us that $k$ must be a multiple of 2. When we plug 5 in for $p$ and $2n$ in for $k$ in (\ref{kcondeq1}) and (\ref{kcondeq2}), we see that there exists a sufficient $r_1$ such that $24|(20 n\mp 4r_1)$. Therefore, it must be that $k\equiv 0 \pmod{2}$.
\item Case 5: Suppose $p\equiv 7\pmod{24}$. Then we note that we can plug 7 in for $p$ in (\ref{combkcond}) to get that \[24 (v_1 + v_p) = 16k.\] This tells us that $k$ must be a multiple of 3. When we plug 7 in for $p$ and $3n$ in for $k$ in (\ref{kcondeq1}) and (\ref{kcondeq2}), we see that there exists a sufficient $r_1$ such that $24|(42n\mp 6r_1)$. Therefore, it must be that $k\equiv 0 \pmod{3}$.
\item Case 6: Suppose $p\equiv 11\pmod 24$. Then we note that we can plug 11 in for $p$ in (\ref{combkcond}) to get that \[24 (v_1 + v_p) = 24k.\] This tells us that $k$ must be an integer. When we plug 11 in for $p$ and $n$ in for $k$ in (\ref{kcondeq1}) and (\ref{kcondeq2}), we see that there exists a sufficient $r_1$ such that $24|(22\mp 10r_1)$. Therefore, it must be that $k\in\Z$.
\item Case 7: Suppose $p\equiv 13 \pmod{24}$. Then we note that we can plug 13 in for $p$ in (\ref{combkcond}) to get that \[24 (v_1 + v_p) = 28k \equiv 4k \pmod{24}.\] This tells us that $k$ must be a multiple of 6. When we plug 13 in for $p$ and $6n$ in for $k$ in (\ref{kcondeq1}) and (\ref{kcondeq2}), we see that there exists a sufficient $r_1$ such that $24|(156n\mp 12r_1)$. Therefore, it must be that $k\equiv 0 \pmod{6}$.
\item Case 8: Suppose $p\equiv 17\pmod{24}$. Then we note that we can plug 17 in for $p$ in (\ref{combkcond}) to get that \[24 (v_1 + v_p) = 36k \equiv 12k \pmod{24}.\] This tells us that $k$ must be a multiple of 2. However, when we plug 17 in for $p$ and $n$ in for $k$ in (\ref{kcondeq1}) and (\ref{kcondeq2}), we get that $24|68n \mp 16r_1$ which cannot be the case since no multiple of 16 is congruent to $10\pmod{24}$. When we consider $k$ to be a multiple of 4 instead, we get that $24|156n \mp 16r_1$. Therefore, it must be that $k\equiv 0 \pmod{4}$.
\item Case 9: Suppose $p \equiv 19\pmod{24}$. Then we note that we can plug 19 in for $p$ in (\ref{combkcond}) to get that \[24 (v_1 + v_p) = 40k \equiv 16k \pmod{24}.\] This tells us that $k$ must be a multiple of 3. When we plug 19 in for $p$ and $3n$ in for $k$ in (\ref{kcondeq1}) and (\ref{kcondeq2}), we see that there exists a sufficient $r_1$ such that $24|(114n\mp 18r_1)$. Therefore, it must be that $k\equiv 0 \pmod{3}$.
\item Case 10: Suppose $p\equiv 23\pmod{24}$. Then we note that we can plug 23 in for $p$ in (\ref{combkcond}) to get that \[24 (v_1 + v_p) = 48k.\] This tells us that $k$ must be an integer. When we plug 23 in for $p$ and $n$ in for $k$ in (\ref{kcondeq1}) and (\ref{kcondeq2}), we see that there exists a sufficient $r_1$ such that $24|(46n\mp 22r_1)$. Therefore, it must be that $k\in \Z$.
\end{itemize}
($\leftarrow$) Suppose that the conditions given hold. We note that we need to show that we can satisfy (\ref{kcondeq1}) and (\ref{kcondeq2}). Since we are suppose sufficient $r_1$ to satisfy the order of vanishing at the cusps, we can just show that our conditions satisfy (\ref{combkcond}). We consider 6 cases. \begin{itemize}
\item Case 1: $p\equiv 11,23 \pmod{24}$. Plugging in our possible $p$ and let $k$ be an integer, we get $24| 24k$ when  we plug in 11, and $24| 48k$ when we plug in 23. Both of which hold true for all integers $k$.
\item Case 2: $p\equiv 2,17 \pmod 24$ and $k\equiv 0 \pmod{4}$. Plugging in our possible $p$ and $k$, we get $24| 24n$ when  we plug in 2, and $24| 144k$ when we plug in 17. Both of which hold true for all integers $n$.
\item Case 3: $p\equiv 3,7,19 \pmod{24}$ and $k \equiv 0\pmod{3}$. Plugging in our possible $p$ and $k$, we get $24| 24n$ when  we plug in 3, $24| 48n$ when we plug in 7, and $24|120n$ when we plug in 19. Both of which hold true for all integers $n$.
\item Case 4: $p\equiv 5 \pmod{24}$ and $k\equiv 0\pmod{2}$. Plugging in our possible $p$ and $k$, we get $24| 24n$. Both of which hold true for all integers $n$.
\item Case 5: $p\equiv 13 \pmod{24}$ and $k\equiv 0\pmod{6}$. Plugging in our possible $p$ and $k$, we get $24| 168n$. Both of which hold true for all integers $n$.
\item Case 6: $p\equiv 1 \pmod{24}$ and $k\equiv 0\pmod{12}$. Plugging in our possible $p$ and $k$, we get $24| 24n$. Both of which hold true for all integers $n$.
\end{itemize} Thus for our given conditions, we get that there is an $\eta$-quotient in $M_k (p,\chi)$ where $\chi$ can be calculated.
\end{proof}

This theorem provides us with the congruence condition on $k$ that needs to be satisfied, but this condition alone is not sufficient for saying that an $\eta$-quotient is in $M_k(\Gamma_1(p))$. However, in combination with the orders of vanishing, we can find a count for the number of $\eta$-quotients.

Consider the system of equations
\begin{equation}\label{vanord}
\begin{pmatrix}1&p\\p&1\end{pmatrix}\begin{pmatrix}r_1\\r_p\end{pmatrix}=\begin{pmatrix}24v_1\\24v_p\end{pmatrix},\end{equation}
where we think of $v_1$ and $v_p$ as the orders of vanishing at the cusps $i\infty$ and $\frac{1}{p}$. Note, we obtained this system of equations by combining our congruence conditions from Theorem \ref{GHNew} and order of vanishing conditions from Theorem \ref{Cvanishing}.

Using this we can view the orders of vanishing as forming the line $v_1+v_p = \frac{k(p+1)}{12}$, shown in the figure. Note, this is explained in more detail in the proof of the next theorem.

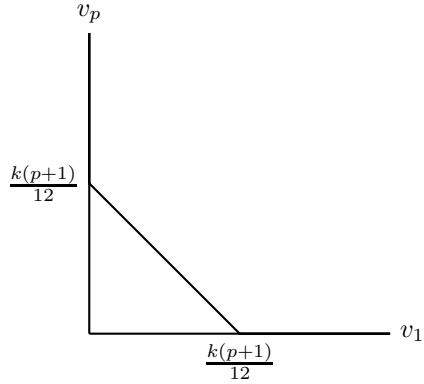
\begin{figure}[h] \begin{center}
  \begin{tikzpicture}[scale=2, thick]
  \draw (0,0) --(2,0) node[midway,below]{$\frac{k(p+1)}{12}$} (2,0)node[right,right]{$v_1$}  -- (1,0) node[midway,below]{} -- (0,1) node[midway,above]{} -- (0,2) node[above,above]{$v_p$}
-- (0,0) node[midway,left]{$\frac{k(p+1)}{12}$};
  \end{tikzpicture}\end{center}
\end{figure}
To obtain a count of $\eta$-quotients which are modular forms, we must consider not only the integer lattice points which lie on the line, but the integer lattice points on the line which correspond to integer $r_1,r_p$ via Equation \ref{vanord}.

Note, it is immediate that cusp forms occur on the interior of the line, and non-cuspidal modular forms occur at the end points. For this reason it is useful to perform the counts of cusp forms and non-cuspidal modular forms separately.

\begin{theorem} \label{etaPcount}
Let $p>3$ be a prime. Let $k=hk'$ where $h$ is the needed divisor of $k$ given by Theorem \ref{kcond}. Let $p-1 = 2hd$ where $d$ is the integer which satisfies this equality. Also, let $c$ be the smallest positive integer representative of $\frac{k'h}{12}$ modulo $d$.
\begin{enumerate}
\item For $c= \frac{k(p+1)}{12} - \floor{\frac{k(p+1)}{12d}} d$, the  number of $\eta$-quotients in $S_k(\Gamma_1(p))$ is \[\frac{k(p+1)}{12d}-1.\]
\item For $c < \frac{k(p+1)}{12} - \floor{\frac{k(p+1)}{12d}} d$, the number of $\eta$-quotients in $S_k(\Gamma_1(p))$ is \[ \ceiling{\frac{k(p+1)}{12d}}.\]
\item For $c > \frac{k(p+1)}{12} - \floor{\frac{k(p+1)}{12d}} d$, the number of $\eta$-quotients in $S_k(\Gamma_1(p))$ is \[ \floor{\frac{k(p+1)}{12d}}.\]
\end{enumerate}
\end{theorem}
\begin{proof}
Let us consider $M_k(p)$ where $k$ is appropriately chosen based on $p$.
We have from Theorem \ref{Cvanishing} the orders of vanishing are given by $24v_1 = (1-p) r_1 + 2pk$, which could be rewritten as $2pk - 24v_1 = (p-1)r_1$. Since by definition $r_1\in\Z$, we have that $2pk-24v_1 \equiv 0 \pmod{p-1}$. From here, we have the following \begin{align*}
2k-24v_1 &\equiv 0 \pmod{p-1}\\
v_1 &\equiv \frac{2h}{24}k' \pmod{d}
\end{align*}
This congruence tells us what $v_1$ needs to be congruent to modulo $d$.
We note that the orders of vanishing sum together to get \[v_1 + v_p = \frac{k(p+1)}{12}.\] This gives us a defined line in $\R^2$. Furthermore, since we are only considering cusp forms, we can say that $v_1$, $v_p >0$. The number of points on this line which satisfy this inequality and the congruence is the number of $\eta$-quotients. We now consider three cases:
\begin{itemize}
\item Case 1: Suppose $0 = \frac{k(p+1)}{12} - \floor{\frac{k(p+1)}{12d}} d$. Then, we have that $v_1 \equiv 0 \pmod{\frac{p-1}{2h}}$. Furthermore, we note that $\frac{p-1}{2h} \left| \frac{k(p+1)}{12}\right.$. Thus, we have that the number of points which match our congruence is $\frac{k(p+1)}{12d}$. However, we will note that one of these points gives us $v_p=0$, which is not desired. Therefore, the number of $\eta$-quotients that are in $S_k(\Gamma_1(p))$ is \[\frac{k(p+1)}{12d}-1.\]
\item Case 2: Suppose $c<\frac{k(p+1)}{12} - \floor{\frac{k(p+1)}{12d}} d$. Then we note that $\floor{\frac{k(p+1)}{12d}}d$ is less than $\frac{k(p+1)}{12}$. However, since $c<\frac{k(p+1)}{12} - \floor{\frac{k(p+1)}{12d}} d$, we have another point to count that is between $\floor{\frac{k(p+1)}{12d}}d$ and $\frac{k(p+1)}{12}$. Therefore, the number of $\eta$-quotients that are in $S_k(\Gamma_1(p))$ is \[\ceiling{\frac{k(p+1)}{12d}}.\]
\item Case 3: Suppose $c>\frac{k(p+1)}{12} - \floor{\frac{k(p+1)}{12d}} d$. Then we note that $\floor{\frac{k(p+1)}{12d}}d$ is less than $\frac{k(p+1)}{12}$. Since $c>\frac{k(p+1)}{12} - \floor{\frac{k(p+1)}{12d}} d$, we have no more points to count between $\floor{\frac{k(p+1)}{12d}}d$ and $\frac{k(p+1)}{12}$. Therefore, the number of $\eta$-quotients that are in $S_k(\Gamma_1(p))$ is \[\floor{\frac{k(p+1)}{12d}}.\]
\end{itemize}
\end{proof}

\begin{coro}
Let $p>3$ be a prime. Let $h$ be the needed divisor of $k$ given by Theorem \ref{kcond}. Then, $M_k(\Gamma_1(p))$ has $\eta$-quotients if \[\frac{p-1}{2h} \leq \frac{k(p+1)}{12}.\]
\end{coro} This follows directly from Theorems \ref{kcond} and \ref{etaPcount}.

We then need to count the number of non-cusp form $\eta$-quotients.
\begin{theorem}
Let $p>3$ be a prime. Then, $M_k(\Gamma_1(p))\setminus S_k(\Gamma_1(p))$ has $\eta$-quotients if and only if $\frac{p-1}{2} | k$. Furthermore, for $k>0$, there are exactly two $\eta$-quotients in $M_k(\Gamma_1(p))\setminus S_k(\Gamma_1(p))$, which are of the form \[ \frac{\etaq{p}{\frac{2pk}{p-1}}}{\etaq{}{\frac{2k}{p-1}}},\quad\text{and}\quad \frac{\etaq{}{\frac{2pk}{p-1}}}{\etaq{p}{\frac{2k}{p-1}}}.\]
\end{theorem}
\begin{proof}
($\rightarrow$) Suppose $f(x) \in M_k(\Gamma_1(p))\setminus S_k(\Gamma_1(p))$ is an $\eta$-quotient satisfying Theorem \ref{GHNew}. Then, we know that at least one of the orders of vanishing must be zero. Thus, we have two cases:
 \begin{itemize}
\item Case 1: Suppose $v_1 = 0$. Then, $pr_1 + r_p = 0$, which can be rewritten to get $(p-1)r_1 = -2k$. Therefore, we have that $\frac{p-1}{2} |k$. Furthermore, we can get that $r_1 = \frac{2k}{p-1}$ and thus, $r_p = \frac{2pk}{p-1}$. When plugging these values into $v_p$ we get \[v_p = \frac{1}{24}\left( \frac{-2k}{p-1} + \frac{2pk}{p-1}\right) = \frac{2k}{24} >0.\]
\item Case 2: Suppose $v_p = 0$. Then, $r_1 + pr_p = 0$, which can be rewritten to get $(1-p)r_1 = -2pk$. Therefore, $\frac{p-1}{2}|k$ since $p\not| (p-1)$ and therefore $p|r_1$. Furthermore, we get that $r_1 = \frac{2pk}{p-1}$, and thus $r_p = \frac{-2k}{p-1}$. When plugging these values into $v_1$ we get \[v_1 = \frac{1}{24}\left( \frac{2pk}{p-1} + \frac{-2k}{p-1}\right) = \frac{2k}{24} > 0.\]
\end{itemize}
In both cases, the number needed to divide $k$ is the same. Furthermore, both create a single $\eta$-quotient for a fixed $k$. Therefore, we have that $\frac{p-1}{2} |k$. Furthermore, there are exactly two $\eta$-quotients which result from looking at either of the orders of vanishing being zero, and they are \[ \frac{\etaq{p}{\frac{2pk}{p-1}}}{\etaq{}{\frac{2k}{p-1}}},\quad\text{and}\quad \frac{\etaq{}{\frac{2pk}{p-1}}}{\etaq{p}{\frac{2k}{p-1}}}.\]

($\leftarrow$) Suppose that $k = \frac{p-1}{2}m>0$ for some integer $m$. Also, suppose we have the two $\eta$-quotients \[ \frac{\etaq{p}{\frac{2pk}{p-1}}}{\etaq{}{\frac{2k}{p-1}}},\quad\text{and}\quad \frac{\etaq{}{\frac{2pk}{p-1}}}{\etaq{p}{\frac{2k}{p-1}}}.\] We consider each $\eta$-quotient as its own case. \begin{itemize}
\item Consider $\frac{\etaq{p}{\frac{2pk}{p-1}}}{\etaq{}{\frac{2k}{p-1}}}$. We note that $r_1+r_p = \frac{-2k}{p-1}+\frac{2pk}{p-1} = 2k$. Furthermore, \[r_1 +pr_p = \frac{-2k}{p-1}+\frac{2pk}{p-1}p = (p^2-1)m \equiv 0 \pmod{24}\] since $p$ is relatively prime to $24$; and \[pr_1+r_p = p\frac{-2k}{p-1}+\frac{2pk}{p-1} = 0 \equiv 0\pmod{24}.\] When looking at that orders of vanishing, we get \[v_1 = \frac{1}{24}(pr_1+r_p) = \frac{1}{24}\left(p\frac{-2k}{p-1}+\frac{2pk}{p-1}\right) = 0 \geq 0,\quad\text{and}\] \[v_p = \frac{1}{24}(r_1 +pr_p) = \frac{1}{24}\left(\frac{-2k}{p-1}+\frac{2pk}{p-1}p\right) = \frac{1}{24}(p^2-1)m\geq 0.\] Since our orders of vanishing are both $\geq 0$ and one of them is equal to $0$, we have that $\frac{\etaq{p}{\frac{2pk}{p-1}}}{\etaq{}{\frac{2k}{p-1}}} \in M_k(\Gamma_1(p))\setminus S_k(\Gamma_1(p))$.
\item Consider $\frac{\etaq{}{\frac{2pk}{p-1}}}{\etaq{p}{\frac{2k}{p-1}}}$. We note that $r_1+r_p = \frac{2pk}{p-1}+\frac{-2k}{p-1} = 2k$. Furthermore, \[r_1 +pr_p = \frac{2pk}{p-1}+p\frac{-2k}{p-1} = 0 \equiv 0 \pmod{24};\] and since $p$ is relatively prime to $24$, \[pr_1+r_p = p\frac{2pk}{p-1}+\frac{-2k}{p-1} = (p^2-1)m \equiv 0\pmod{24}.\] When looking at that orders of vanishing, we get \[v_1 = \frac{1}{24}(pr_1+r_p) = \frac{1}{24}\left(p\frac{2pk}{p-1}+\frac{-2k}{p-1}\right) = \frac{1}{24}(p^2-1)m \geq 0,\quad\text{and}\] \[v_p = \frac{1}{24}(r_1 +pr_p) = \frac{1}{24}\left(\frac{2pk}{p-1}+\frac{-2k}{p-1}p\right) = 0\geq 0.\] Since our orders of vanishing are both $\geq 0$ and one of them is equal to $0$, we have that $\frac{\etaq{}{\frac{2pk}{p-1}}}{\etaq{p}{\frac{2k}{p-1}}} \in M_k(\Gamma_1(p))\setminus S_k(\Gamma_1(p))$.
\end{itemize} Thus, $M_k(\Gamma_1(p))\setminus S_k(\Gamma_1(p))$ has $\eta$-quotients.
\end{proof}
From the $\eta$-quotients given in the theorem, let $k=\frac{p-1}{2}m$ where $m$ is a positive integer. Then the $\eta$-quotients have characters \[\chi_1(n) = \left(\frac{(-1)^{\frac{p-1}{2}m}p^{pm}}{n}\right),\quad\text{and}\quad \chi_2(n) = \left(\frac{(-1)^{\frac{p-1}{2}m}p^m}{n}\right),\] respectively. In the case where $m$ is even, both of the characters are guaranteed to be the trivial character. When $m$ is odd, we are guaranteed to have quadratic character.

Now that we know how many $\eta$-quotients there are and can write down what they are if needed. Now, we need to know the dimension of the space spanned by these $\eta$-quotients.
\begin{theorem}\label{plinind}
Let $p>3$ be a prime. Then, the $\eta$-quotients in $M_k(\Gamma_1(p))$ given by the previous theorems are linearly independent.
\end{theorem}
\begin{proof}
Suppose that we are looking at $\eta$-quotients in $M_k(\Gamma_1(p))$ for $p>3$ a prime. Without loss of generality, we look at the Fourier series given by the point at $\infty$. By using the Sturm Bound, we get that we need to compare the first $\floor{\frac{pk}{12}}+1$ terms of each Fourier series. We can pick a cusp and order the $\eta$-quotients increasingly by looking at the order of vanishing. We can then create a matrix, $A$, where the $i,j$-th entry represents $a(j)$ in the $i$-th $\eta$-quotient's Fourier series. Since all of the $\eta$-quotients have different orders of vanishing and they are in increasing order, we get that $A$ is in echelon form. Furthermore, none of the rows change to non-zero  at the same time. This tells us that all the rows are linearly independent. Thus all of the $\eta$-quotients are linearly independent.
\end{proof}
The following corollaries can all be obtained by comparing comparing dimension formulas with our counts and applying the previous theorem.
\begin{coro}
Let $p\geq 5$ be a prime. Denote the space of level $p$, weight $k$ $\eta$-quotients by $\eta_k(p)$. Then, 
\begin{enumerate}
\item
If $p\equiv 3\pmod{4}$, then taking the limit over odd $k$ in the appropriate congruence class from Theorem \ref{kcond} we have
$$\lim_{k\rightarrow \infty}\frac{\dim\eta_k(p)}{\dim S_k\left(p,\left(\frac{\cdot}{p}\right)\right)}=\frac{2h}{p-1},$$
\item
If $p\equiv 3\pmod{4}$, then taking the limit over even $k$ in the appropriate congruence class from Theorem \ref{kcond} we have
$$\lim_{k\rightarrow \infty}\frac{\dim\eta_k(p)}{\dim S_k(\Gamma_0(p))}=\frac{2h}{p-1}.$$
\item
If $p\equiv 1\pmod{4}$, then taking the limit over all $k$ in the appropriate congruence class from Theorem \ref{kcond} we have
$$\lim_{k\rightarrow \infty}\frac{\dim\eta_k(p)}{\dim S_k(\Gamma_0(p))+\dim S_k\left(p,\left(\frac{\cdot}{p}\right)\right)}=\frac{h}{p-1}.$$
\end{enumerate}
\end{coro}
Finally, we would like to consider the case that our $v_1$ and $v_p$ are integral but do not correspond to integral $r_1,r_p$. To gain some intuition concerning the properties of the ``$\eta$-quotients'' formed from these $r_1,r_p$ we consider the following example.
\begin{example}
Let $p=11$ and $k=6$. Note, in this situation we have that in order to have $\eta$-quotients we must have $v_1\equiv 3\pmod{5}$. So, we will investigate the properties of the function obtained by choosing $v_1\not\equiv 3\pmod{5}$.

Consider $v_1=1$. This implies $v_p=5$.
Then,
$$\begin{pmatrix}r_1\\r_p\end{pmatrix}=\begin{pmatrix}1&11\\11&1\end{pmatrix}^{-1}\begin{pmatrix}24\\120\end{pmatrix}=\begin{pmatrix}54/5\\6/5\end{pmatrix}.$$
Now, we can use these to form the ``$\eta$-quotient''
$$f(z)=\eta^{\frac{54}{5}}(z)\eta^{\frac{6}{5}}(11z).$$
Using the transformation properties from Section \ref{EtaQuotient} we have that
\begin{align*}
f(Tz)&=e^{\frac{27\pi i}{30}}\eta^{\frac{54}{5}}(z)e^{\frac{11\pi i }{10}}\eta^{\frac{6}{5}}(11z)\\
&=f(z).
\end{align*}
Note, if we raise $f(z)$ to the $5^{th}$ power to cancel the denominators of the $r_1$ and $r_p$ then we can use Theorem \ref{GHNew} to verify that we obtain $f(z)^5\in S_30(\Gamma_0(11))$, i.e., our lattice point corresponds to a ``root'' of an $\eta$ quotient of higher weight.

Note, the remaining choices for $v_1$ give us similar results.
\end{example}

\subsection{$N=pq$}
For this section, we shall assume that the level of the modular forms being discussed is always the product of two distinct primes greater than 3.

To start we wish to place conditions on $k$ to guarantee a weakly holomorphic modular form.
\begin{theorem}
For $N=pq$, if $k$ is a positive integer satisfying \[\gcd\left(\frac{p-1}{2},\frac{q-1}{2},\frac{pq-1}{2},12\right)|k,\]
then $M^{!}_k(\Gamma_1(N))$ contains $\eta$-quotients.
\end{theorem}
This is a specialization of Theorem \ref{sqfkcond}, and the proof can be found with the theorem. We note that the sum of all the orders of vanishing is \[v_1+v_p+v_q+v_{pq} = \frac{k(pq+p+q+1)}{12}.\] 
Thus, in order to have modular forms our orders of vanishing must satisfy $0\leq v_\delta \leq \frac{k(pq+p+q+1)}{12}$, for all $\delta|N$. Furthermore, we have that each lattice point on the hyperplane given by the previous summation corresponds to an $\eta$-quotient which may have non-integral exponents. For a fixed set of vanishing orders we can calculate the $\eta$-quotient by computing \[\begin{bmatrix}r_1\\r_p\\r_q\\2k\end{bmatrix} = \begin{bmatrix}
pq-1 & q-1 & p-1 & 1\\
q-p & pq-p & 1-p & p\\
p-1 & 1-q & pq-q & q\\
1-pq & p-pq & q-pq & pq
\end{bmatrix}^{-1}\begin{bmatrix}24v_1\\24v_p\\24v_q\\24v_{pq}\end{bmatrix}.\] Thus the number of lattice points on the hyperplane provides an upper bound for the number of $\eta$-quotients. However, we need to guarantee that $r_\delta\in\mathbb{Z}$ for each $\delta|N$. We can write the following equations 
\begin{align}
24v_1 &= (pq-1)r_1 + (q-1)r_p + (p-1)r_q + 2k\label{pqv1}\\
24v_p &= (q-p)r_1 + (pq-p)r_p + (1-p)r_q + 2kp\label{pqvp}\\
24v_q &= (p-q)r_1 + (1-q)r_p + (pq-q)r_q + 2kq\label{pqvq}\\
24v_{pq} &= (1-pq)r_1 + (p-pq)r_p + (q-pq)r_q + 2kpq\label{pqvpq}
\end{align}
We get the following congruence by creating linear combinations of the equations above. From  (\ref{pqv1}) and (\ref{pqvp}), we have $24(v_1+v_p) \equiv 0 \pmod{p+1}$. From  (\ref{pqv1}) and (\ref{pqvq}), we have $24(v_1+v_q) \equiv 0 \pmod{q+1}$. From  (\ref{pqvp}) and (\ref{pqvpq}), we have $24(v_p+v_{pq}) \equiv 0 \pmod{q+1}$. From  (\ref{pqvq}) and (\ref{pqvpq}), we have $24(v_1+v_{pq}) \equiv 0 \pmod{p+1}$. From (\ref{pqv1}) and (\ref{pqvpq}), we have $24(v_1+v_{pq}) \equiv 2k(1+pq)\pmod{(p-1)(q-1)}$. From (\ref{pqvp}) and (\ref{pqvq}), we have $24(v_p+v_q) \equiv 2k(p+q)\pmod{(p-1)(q-1)}$. While satisfying these congruences still does not guarantee a modular form, the number of lattice points which do satisfy the congruences provide a better upper bound on the number of $\eta$-quotients than the previous upper bound.

\subsection{Square-Free $N$}
For this section, we will assume that $N$ is odd and square-free. We begin by finding conditions on $k$ that are similar to Theorem \ref{kcond}
\begin{theorem}\label{sqfkcond}
Fix a level $N=p_1\cdots p_s$ and a weight $k$. If $k$ satisfies \[ \left. \gcd \left( \gcd_{\substack{\delta |N\\ \delta \neq N}}\left(\frac{N}{d} - 1\right),24\right) \right| 2k\] then $M_k^{!}(\Gamma_1(N))$ contains $\eta$-quotients.
\end{theorem}
\begin{proof}
($\rightarrow$) Suppose that $f$ is an $\eta$-quotient satisfying Theorem \ref{GHNew}. Then, we can write that $2k = \sum_{\delta |N} r_\delta$; which can be written as \[r_N = 2k - \sum_{\substack{\delta |N\\ \delta\neq N}} r_\delta.\] Furthermore, we have that \[\sum_{\delta |N} \frac{N}{\delta} r_\delta = -24 M,\] where $r_\delta, M\in\Z$. We can rewrite this last equation as \[ 24M + \sum_{\substack{\delta |N\\ \delta\neq N}} \left(\frac{N}{\delta} - 1\right) r_\delta = -2k.\] Since $M,r_\delta \in\Z$, we have that $\gcd \left( \gcd_{\substack{\delta |N\\ \delta \neq N}}\left(\frac{N}{d} - 1\right),24\right) | 2k$.\\
($\leftarrow$) Let $f(z) = \prod_{\delta |N}\etaq{\delta}{r_\delta}$. Also, let $d=\gcd \left( \gcd_{\substack{\delta |N\\ \delta \neq N}}\left(\frac{N}{d} - 1\right),24\right)$. Since $d|2k$, let $-2k=dW$. Let $x,y_\delta\in \Z$. We can write the following \begin{align*}
d &= 24x + \sum_{\delta |N} \left(\frac{N}{\delta}-1\right)r_\delta\\
-2k &= 24xW + \sum_{\delta |N} \left(\frac{N}{\delta}-1\right)r_\delta W\\
24xW &= 2k + \sum_{\delta |N} \left(\frac{N}{\delta}-1\right)r_\delta W\\
\sum_{\delta |N} \frac{N}{\delta}r_\delta W &\equiv 0 \pmod{24}
\end{align*} This gives us our first congruence. If we multiply our sum by $N$, we get $\sum_{\delta |N} \delta r_\delta W \equiv 0 \pmod{24}$, Giving us our second congruence condition. Therefore, $f(z)$ is a weakly holomorphic modular form.
\end{proof}

Similar to the $N=pq$ case, we can sum together all of the orders of vanishing to get \begin{align*}
\sum_{d |N} v_d &= \frac{1}{24}\sum_{d,\delta |N} \frac{N(d,\delta)^2r_\delta}{(d\delta)}\\
&= \frac{1}{24}\sum_{\delta |N} \sigma_1(N)r_\delta\\
&= \frac{k\sigma_1(N)}{12}
\end{align*} This gives us a hyperplane in $\R^{\sigma_0(N)}$. The number of lattice points in the hyperplane whose entries are all nonnegative provides an upper bound on the number of $\eta$-quotients.

\subsection{$N=4p$}
In this section, we shall asssume that the level we are looking at is $4p$ for a prime $p>3$.
\begin{prop}\label{p2kprop}
For all $k\geq 0$, \[ \frac{\eta^{4k}(8z)}{\eta^{2k}(4z)} \in M_k(8,\chi) \] for $\chi(n) = \left(\frac{(-1)^k}{n}\right)$.
\end{prop}
\begin{proof}
First, we have $4k+(-2)k = 2k$, as desired. We also have $(-2) 4k + 8 (4k) = 24k \equiv 0 \pmod{24}$, as desired; and $2(-2k) + 4k = 0 \equiv 0 \pmod{24}$, as desired. Finally, \begin{align*}
\frac{1}{4}(-2k) + \frac{1}{8} (4k) &= 0\\
(-2k) + \frac{1}{2} (4k) &= 0\\
4(-2k) + 2(4k) &= 0\\
4(-2k) + 8(4k) &= 24k \geq 0\\
\end{align*}
\end{proof}

\begin{prop}\label{pekprop}
Let $p$ be a prime. Also, let $k$ be an even positive integer. Then, \[ \frac{\eta^{4k}(4pz)}{\eta^{2k}(4z)}  \in M_k (\Gamma_0(4p)).\]
\end{prop}
\begin{proof}
Note that the $p=2$ case is shown in Proposition \ref{p2kprop}. Therefore, assume $p\geq 3$. (Sketch) $4k+(-2)k = 2k$ as desired. $2p(-2k) + 4p(4k) = 12pk \equiv 0 \pmod{24}$, since $k$ is even, as desired. $2(-2k) + 4k = 0 \equiv 0 \pmod{24}$ as desired. Finally, \begin{align*}
\frac{1}{2p}(-2k) + \frac{1}{4p} (4k) &= 0\\
\frac{2}{p}(-2k) + \frac{1}{p} (4k) &= 0\\
\frac{2}{p}(-2k) + \frac{4}{p}(4k) &= \frac{12}{p}k \geq 0\\
\frac{p}{2}(-2k) + \frac{p}{4}(4k) &= 0\\
2p(-2k) + p(4k) &= 0\\
2p(-2k) + 4p(4k) &= 12pk \geq 0\\
\end{align*}
\end{proof}

\section{Future Work}
As previously mentioned, all of the congruence conditions for $N=pq$ and square-free $N$ do not provide sufficient conditions to give an $\eta$-quotient. In the future, it would be nice to figure out the final conditions needed to yield an $\eta$-quotient. It would also be important to study the case $N=p^2$ and extend the cases to $p^n$. Assuming that there are results here, the results could be combined to look at the general $N$ case.

We can also expand Theorem \ref{plinind} to the following conjecture:
\begin{conjecture}
All $\eta$-quotients gained from Theorems \ref{GHNew} and \ref{Cvanishing} are linearly independent.
\end{conjecture} While, we are not sure if this will hold for all $N$, we are confident that this should hold for square-free $N$. If this conjecture holds and we are able to obtain a count for the number of $\eta$-quotients for $\Gamma_1(N)$, then we can figure out how $\eta$-quotients span a subspace of modular forms at higher levels.
\bigskip\bigskip

\bibliographystyle{alpha}
\bibliography{bibliography}
\end{document}